\documentclass [12pt]{amsart}
\usepackage{graphicx}
\usepackage{geometry}
\usepackage{amsthm}
\usepackage{amssymb}

\theoremstyle{plain}
\newtheorem{thm}{Theorem}[section]
\newtheorem{cor}[thm]{Corollary}
\newtheorem{lem}[thm]{Lemma}
\newtheorem{prop}[thm]{Proposition}
\newtheorem{defn}[thm]{Definition}
\newtheorem{exa}[thm]{Example}
\newtheorem{rem}[thm]{Remark}
\newtheorem{ques}[thm]{Question}

\begin{document}

\title [{{Graded Weakly Prime Ideals of Non-Commutative Rings}}]{Graded Weakly Prime Ideals of Non-Commutative Rings}

\author[{{A. S. Alshehry }}]{\textit{Azzh Saad Alshehry }}

\address
{\textit{Azzh Saad Alshehry, Department of Mathematical Sciences, Faculty of Sciences,
Princess Nourah Bint Abdulrahman University, P.O. Box 84428, Riyadh 11671, Saudi Arabia.}}
\bigskip
{\email{\textit{asalshihry@pnu.edu.sa}}}

 \author[{{R. Dawwas }}]{\textit{Rashid Abu-Dawwas }}

\address
{\textit{Rashid Abu-Dawwas, Department of Mathematics, Yarmouk
University, Irbid, Jordan.}}
\bigskip
{\email{\textit{rrashid@yu.edu.jo}}}

 \subjclass[2010]{16W50, 13A02}

\date{}

\begin{abstract} In this article, we consider the structure of graded rings, not necessarily commutative nor with unity, and study the graded weakly prime ideals. We investigate the graded rings in which all graded ideals are graded weakly prime. Several properties are given, and several examples to support given propositions are constructed. We initiate the study of graded weakly total prime ideals and investigate graded rings for which every proper graded ideal is graded weakly total prime.
\end{abstract}

\keywords{Graded weakly prime ideals, graded weakly total prime ideals, graded total prime ideals, graded prime ideals.
 }
 \maketitle

\section{Introduction}

Throughout this article, all rings are not necessarily commutative nor with unity. Let $G$ be a group with identity $e$. Then a ring $R$ is said to be $G$-graded if $R=\displaystyle\bigoplus_{g\in G} R_{g}$ with $R_{g}R_{h}\subseteq R_{gh}$ for all $g, h\in G$ where $R_{g}$ is an additive subgroup of $R$ for all $g\in G$; here $R_{g}R_{h}$ denotes the additive subgroup of $R$ consisting of all finite sums of elements $r_{g}s_{h}$ with $r_{g}\in R_{g}$ and $s_{h}\in R_{h}$. We denote this by $G(R)$. It is clear that $R_{g}$ is an $R_{e}$-module for all $g\in G$. The elements of $R_{g}$ are called homogeneous of degree $g$. If $x\in R$, then $x$ can be written as $\displaystyle\sum_{g\in G}x_{g}$, where $x_{g}$ is the component of $x$ in $R_{g}$. Also, we set $h(R)=\displaystyle\bigcup_{g\in G}R_{g}$. Moreover, it has been proved in \cite{Nastasescue} that $R_{e}$ is a subring of $R$ and if $R$ has a unity $1$, then $1\in R_{e}$. Let $I$ be an ideal of a graded ring $R$. Then $I$ is said to be a graded ideal if $I=\displaystyle\bigoplus_{g\in G}(I\cap R_{g})$, i.e., for $x\in I$, $x=\displaystyle\sum_{g\in G}x_{g}$, where $x_{g}\in I$ for all $g\in G$. The following example shows that an ideal of a graded ring need not be graded.

\begin{exa}\label{A}(\cite{Dawwas Bataineh Muanger}) Consider $R=M_{2}(K)$ (the ring of all $2\times2$ matrices with entries from a field $K$) and $G=\mathbb{Z}_{4}$ (the group of integers modulo $4$). Then $R$ is $G$-graded by
\begin{center}
$R_{0}=\left(
   \begin{array}{cc}
     K & 0 \\
     0 & K \\
   \end{array}
 \right)$, $R_{2}=\left(
                    \begin{array}{cc}
                      0 & K \\
                      K & 0 \\
                    \end{array}
                  \right)$ and $R_{1}=R_{3}=\{0\}$.
\end{center} Consider the ideal $I=\langle\left(
                                                    \begin{array}{cc}
                                                      1 & 1 \\
                                                      1 & 1 \\
                                                    \end{array}
                                                  \right)\rangle$ of $R$. Note that, $\left(
                                                                                 \begin{array}{cc}
                                                                                   1 & 1 \\
                                                                                   1 & 1 \\
                                                                                 \end{array}
                                                                               \right)\in I$
such that $\left(
             \begin{array}{cc}
               1 & 1 \\
               1 & 1 \\
             \end{array}
           \right)=\underbrace{\left(
                                 \begin{array}{cc}
                                   1 & 0 \\
                                   0 & 1 \\
                                 \end{array}
                               \right)}_{\in R_{0}}+\underbrace{\left(
                                                                   \begin{array}{cc}
                                                                     0 & 1 \\
                                                                     1 & 0 \\
                                                                   \end{array}
                                                                 \right)}_{\in R_{2}}$.
If $I$ is a graded ideal of $R$, then $\left(
                                         \begin{array}{cc}
                                           1 & 0 \\
                                           0 & 1 \\
                                         \end{array}
                                       \right)\in I$ which is a contradiction. So, $I$ is not graded ideal of $R$.
\end{exa}

\begin{lem} \label{B}(\cite{Dawwas Bataineh Muanger}) Let $R$ be a graded ring.

\begin{enumerate}

\item If $I$ and $J$ are graded ideals of $R$, then $I+J$ and $I\bigcap J$ are graded ideals of $R$,

\item If $x\in h(R)$, then $Rx$ is a graded ideal of $R$.
\end{enumerate}
\end{lem}

Graded prime ideals of graded commutative rings with unity have been introduced and studied in \cite{Refai Hailat Obiedat}. A proper graded ideal $P$ of a graded commutative ring $R$ with unity is said to be graded prime if whenever $x, y\in h(R)$ such that $xy\in P$, then either $x\in P$ or $y\in P$. The concept of graded prime ideals and its generalizations have an outstanding location in graded commutative algebra. They are valuable tools to determine the properties of graded commutative rings. Various generalizations of graded prime ideals have been studied. Indeed, Atani introduced in \cite{Atani Prime} the concept of graded weakly prime ideals. A proper graded ideal $P$ of a graded commutative ring $R$ with unity is said to be a graded weakly prime ideal of $R$ if whenever $x, y\in h(R)$ such that $0\neq xy\in P$, then $x\in P$ or $y\in P$. In \cite{Dawwas Bataineh Muanger}, graded prime ideals of graded non-commutative rings have been studied. In this article, we consider the structure of graded rings, not necessarily commutative nor with unity, and study the graded weakly prime ideals.

By Theorem 2.12 of Atani \cite{Atani Prime}, the following statements are equivalent for a graded ideal $P$ of $G(R)$ with $P\neq R$ where $R$ is a graded commutative ring with unity:

\begin{enumerate}
\item $P$ is a graded weakly prime ideal of $G(R)$.

\item For each $g, h\in G$, the inclusion $0\neq IJ\subseteq P$ with $R_{e}$-submodules $I$ of $R_{g}$ and $J$ of $R_{h}$ implies that $I\subseteq P$ or $J\subseteq P$.
\end{enumerate}

For graded rings that are not necessarily commutative, it is clear that (2) does not imply (1). So, we define a graded ideal of $G(R)$ to be graded weakly prime as follows: a graded ideal $P$ of $G(R)$ with $P\neq R$ is said to be a graded weakly prime ideal of $G(R)$ if for each $g, h\in G$, the inclusion $0\neq IJ\subseteq P$ with $R_{e}$-submodules $I$ of $R_{g}$ and $J$ of $R_{h}$ implies that $I\subseteq P$ or $J\subseteq P$. In \cite{Dawwas Bataineh Muanger}, the standard definition of a graded prime ideal $P$ for a graded noncommutative ring $R$ is that $P\neq R$ and whenever $I$ and $J$ are graded ideals of $R$ such that $IJ\subseteq P$, then either $I\subseteq P$ or $J\subseteq P$. Accordingly, we define a graded ideal of a graded ring $R$ to be graded weakly prime as follows: a proper graded ideal $P$ of $R$ is said to be a graded weakly prime ideal of $R$ if whenever $I$ and $J$ are graded ideals of $R$ such that $0\neq IJ\subseteq P$, then either $I\subseteq P$ or $J\subseteq P$.

Note that by definition, a graded weakly prime ideal is a proper graded ideal of a graded ring. It is therefore not possible that every graded ideal of a graded ring is a graded weakly prime ideal. However, a graded ring whose zero ideal is graded prime is called a graded prime ring. In this sense, every graded ring is a graded weakly prime ring since the zero ideal is always graded weakly prime. We may therefore say that every graded ideal of a graded ring is graded weakly prime when every proper graded ideal of the graded ring is a graded weakly prime ideal. In this article, we investigate the graded rings in which all graded ideals are graded weakly prime.

In \cite{Dawwas Bataineh Muanger}, a proper graded ideal $P$ of a graded ring $R$ is a graded total prime ideal if $xy\in P$ implies $x\in P$ or $y\in P$ for $x, y\in h(R)$. In this article, we introduce the concept of graded weakly total prime ideals. A proper graded ideal $P$ of a graded ring $R$ is a graded weakly total prime ideal if $0\neq xy\in P$ implies $x\in P$ or $y\in P$ for $x, y\in h(R)$. We investigate graded rings for which every proper graded ideal is graded weakly total prime.

\section{Graded Weakly Prime Ideals}

In this section, we study the structure of graded rings, not necessarily commutative nor with unity, in which every graded ideal is graded weakly prime.

\begin{defn}
\begin{enumerate}
\item A proper graded ideal $P$ of $R$ is said to be a graded weakly prime ideal of $R$ if whenever $I$ and $J$ are graded ideals of $R$ such that $0\neq IJ\subseteq P$, then either $I\subseteq P$ or $J\subseteq P$.

\item A graded ideal of $G(R)$ with $P\neq R$ is said to be a graded weakly prime ideal of $G(R)$ if for each $g, h\in G$, the inclusion $0\neq IJ\subseteq P$ with $R_{e}$-submodules $I$ of $R_{g}$ and $J$ of $R_{h}$ implies that $I\subseteq P$ or $J\subseteq P$.
\end{enumerate}
\end{defn}

Our early proposition is Proposition 2.2 of Atani \cite{Atani Prime} in a further general scene.

\begin{prop}\label{Proposition 1(1)}Let $P$ be a graded weakly prime ideal of $R$. If $P$ is not graded prime ideal of $R$, then $P^{2} = 0$.
\end{prop}

\begin{proof}Since $P$ is graded weakly prime ideal of $R$ which is not prime, there exist graded ideals $I\nsubseteq P$ and $J\nsubseteq P$ with $0 = IJ\subseteq P$. If $P^{2}\neq 0$, then $0\neq P^{2}\subseteq (I + P)(J + P)\subseteq P$, which implies that either $I\subseteq P$ or $J\subseteq P$, which is a contradiction. Hence, $P^{2} = 0$.
\end{proof}

\begin{prop}\label{Proposition 2(1)}Let $R$ be a graded ring with unity and $P$ be a graded ideal of $G(R)$ with $P\neq R$. Then $P$ is a graded weakly prime ideal of $G(R)$ if and only if whenever $x, y\in h(R)$ such that $0\neq xRy\subseteq P$, then either $x\in P$ or $y\in P$.
\end{prop}

\begin{proof}Suppose that $P$ is a graded weakly prime ideal of $G(R)$. Let $x, y\in h(R)$ such that $0\neq xRy\subseteq P$. Then $x\in R_{g}$ and $y\in R_{h}$ for some $g, h\in G$, and then $R_{e}x$ is an $R_{e}$-submodule of $R_{g}$ and $R_{e}y$ is an $R_{e}$-submodule of $R_{h}$. Since $R$ has unity, $0\neq(R_{e}x)(R_{e}y)\subseteq P$, whence $x\in R_{e}x\subseteq P$ or $y\in R_{e}y\subseteq P$. Conversely, assume that $g, h\in G$. Suppose that $IJ\subseteq P$ for $R_{e}$-submodule $I$ of $R_{g}$ and $R_{e}$-submodule $J$ of $R_{h}$, where $I\nsubseteq P$ and $J\nsubseteq P$. Let $x\in I-P$, $y\in J-P$, $z\in I\bigcap P$, and $w\in J\bigcap P$. Since $x + z, y + w\notin P$, we should have $0 = (x + z)R(y + w)$. Considering all combinations where $z$ and/or $w$ equal zero shows that $0 = xy = xw = zy = zw$, and hence $IJ = 0$.
\end{proof}

We are interested in the structure of graded rings in which every graded ideal is graded weakly prime. Note that by definition, a graded weakly prime ideal is a proper graded ideal of a graded ring. It is therefore not possible that every graded ideal of a graded ring is a graded weakly prime ideal. However, a graded ring whose zero ideal is graded prime is called a graded prime ring. In this sense, every graded ring is a graded weakly prime ring since the zero ideal is always graded weakly prime. We may therefore say that every graded ideal of a graded ring is graded weakly prime when every proper graded ideal of the graded ring is a graded weakly prime ideal. If $R^{2} = 0$, then it is clear that every graded ideal of $R$ is graded weakly prime. In particular, if a graded ideal $P$ of a graded ring $R$ is graded weakly prime but not a graded prime ideal, then every graded ideal of $P$ as a graded ring is graded weakly prime by Proposition \ref{Proposition 1(1)}.

\begin{prop}\label{Proposition 3(1)}Every graded ideal of a graded ring $R$ is graded weakly prime if and only if for any graded ideals $P$ and $Q$ of $R$, $PQ = P$, $PQ = Q$, or $PQ = 0$.
\end{prop}

\begin{proof}Suppose that every graded ideal of $R$ is graded weakly prime. Let $P, Q$ be graded ideals of $R$. If $PQ\neq R$, then $PQ$ is graded weakly prime. If $0\neq PQ\subseteq PQ$, then we have $P\subseteq PQ$ or
$Q\subseteq PQ$, that is, $P = PQ$ or $Q = PQ$. If $PQ = R$, then we have $P = Q = R$ whence $R^{2} = R$. Conversely, let $K$ be a proper graded ideal of $R$ and suppose that $0\neq PQ\subseteq K$ for graded ideals $P$ and $Q$ of $R$. Then we have either $P = PQ\subseteq K$ or $Q = PQ\subseteq K$.
\end{proof}

\begin{cor}\label{Corollary 1(1)}Let $R$ be a graded ring in which every graded ideal of $R$ is graded weakly prime. Then for any graded ideal $P$ of $R$, either $P^{2} = P$ or $P^{2} = 0$.
\end{cor}

The next example shows that the converse of Corollary \ref{Corollary 1(1)} is not true in general.

\begin{exa}\label{Example 1(1)}Let $K$ be a field and $R=K\bigoplus K\bigoplus K$. Then every ideal of $R$ is idempotent. So, if $R$ is $G$-graded by any group $G$, then every graded ideal of $R$ is idempotent. On the other hand, if we consider the trivial graduation for $R$ by a group $G$, then $P=K\bigoplus0\bigoplus0$ will be a graded ideal of $R$ which is not graded weakly prime.
\end{exa}

Suppose that a graded ring $R$ with unity has a graded maximal ideal $X$ and $X^{2} = 0$. Thus, the product of any two graded ideals contained in $X$ is zero. It is obvious that every proper graded ideal of $R$ is contained in $X$, and for any graded ideal $P$, $PR = RP = P$. Hence, every graded ideal of $R$ is graded weakly prime. In this case, note that, $X$ is the only graded prime ideal of $R$. In particular, Corollary \ref{Corollary 1(1)} yields that if a graded ring $R$ has the property that every graded ideal is graded weakly prime, then either $R^{2} = R$, or $R^{2} = 0$. Note that, $R^{2}$ is neither $0$ nor $R$ in the next example.

\begin{exa}\label{Example 2(1)}Let $S$ be a ring such that $S^{2}=0$, and let $K$ be a field. Suppose that $R=K\bigoplus S\bigoplus S$ and $G=\mathbb{Z}_{2}$. Then $R$ is $G$-graded by $R_{0}=K\bigoplus 0\bigoplus 0$ and $R_{1}=0\bigoplus S\bigoplus S$. Now, $X=R_{1}$ is a graded maximal ideal of $R$ and $X^{2}=0$. Consider the ideal $P=K\bigoplus0\bigoplus S$ of $R$. To prove that $P$ is graded ideal, let $x\in P$. Then $x=a+0+s$ for some $a\in K$ and $s\in S$, and then $x=(a+0+0)+(0+0+s)=x_{0}+x_{1}$ where $x_{0}=a+0+0\in R_{0}\bigcap P$ and $x_{1}=0+0+s\in R_{1}\bigcap P$. Hence, $P$ is a graded ideal of $R$. Similarly, $I=K\bigoplus S\bigoplus0$ is a graded ideal of $R$, and $0\neq I^{2}\subseteq P$ but $I\nsubseteq P$, which means that $P$ is not a graded weakly prime ideal of $R$.
\end{exa}

If a graded ring $R$ satisfying $R^{2} = R$ has a graded maximal ideal $X$ and $X^{2} = 0$, then every proper graded ideal of $R$ is contained in $X$. However, it is possible that $XR\neq X$. Thus, such a graded ring does not necessarily have the property that every graded ideal is graded weakly prime, see the following example.

\begin{exa}\label{Example 3(1)}Let $K$ be a field, $S=\left(\begin{array}{ccc}
                                                                     0 & K & K \\
                                                                     0 & K & K \\
                                                                     0 & 0 & 0
                                                                   \end{array}
\right)$ and $G=\mathbb{Z}_{2}$. Then $S$ is $G$-graded ring by $S_{0}=\left(\begin{array}{ccc}
                                                                                       0 & 0 & K \\
                                                                                       0 & K & 0 \\
                                                                                       0 & 0 & 0
                                                                                     \end{array}
\right)$ and $S_{1}=\left(\begin{array}{ccc}
                            0 & K & 0 \\
                            0 & 0 & K \\
                            0 & 0 & 0
                          \end{array}
\right)$. Now, $S$ has a unique graded maximal ideal $L=\left\{\left(\begin{array}{ccc}
                                                                       0 & x & y \\
                                                                       0 & 0 & z \\
                                                                       0 & 0 & 0
                                                                     \end{array}
\right):x, y, z\in K\right\}$. Consider the graded ideal $P=\left\{\left(\begin{array}{ccc}
                                                                       0 & 0 & y \\
                                                                       0 & 0 & 0 \\
                                                                       0 & 0 & 0
                                                                     \end{array}
\right):y\in K\right\}$ of $S$. Then $R=S/P$ is a graded ring by $R_{j}=(S_{j}+P)/P$ for $j=0, 1$. Note that, $R^{2}=R$ and $X=L/P$ is a graded maximal ideal of $R$ whose square is zero, the proper graded ideals $RX$ and $XR$ are not graded weakly prime.
\end{exa}

\begin{prop}\label{Proposition 4(1)}Let $R$ be a graded ring such that every graded ideal of $R$ is graded weakly prime and $R^{2} = R$. Then $R$ has at most two graded maximal ideals.
\end{prop}

\begin{proof}Suppose that $R$ has three distinct graded maximal ideals, say $X_{1}$, $X_{2}$, and $X_{3}$. Then $X_{1}X_{2}\neq 0$. Thus we have $0\neq X_{1}X_{2}\subseteq X_{1}\bigcap X_{2}$. But this implies either $X_{1}\subseteq X_{2}$ or $X_{2}\subseteq X_{1}$, a contradiction.
\end{proof}

The next example shows that the condition $R^{2} = R$ in Proposition \ref{Proposition 4(1)} is necessary.

\begin{exa}\label{Example 4(1)}Let $X$ be the unique maximal ideal of $\mathbb{Z}_{4}$. Then $R=X\bigoplus X\bigoplus X$ is a ring trivially graded by a group $G$ whose all graded ideals are graded weakly prime and having more than two graded maximal ideals.
\end{exa}

\begin{prop}\label{Proposition 5(1)}Let $R$ be a graded ring such that every graded ideal of $R$ is graded weakly prime. If $R$ has two graded maximal ideals, then their product is zero. Moreover, if $R$ has a unity, then $R$ is a direct sum of two graded simple rings.
\end{prop}

\begin{proof}Suppose that $R$ has two distinct graded maximal ideals $X_{1}$ and $X_{2}$. Then since $X_{1}\bigcap X_{2}$ is graded weakly prime and $X_{1}X_{2}\subseteq X_{1}\bigcap X_{2}$, we have $X_{1}X_{2} = 0$, and similarly $X_{2}X_{1} = 0$. If $R$ has a unity, then $X_{1}\bigcap X_{2} = \left(X_{1}\bigcap X_{2}\right)R = \left(X_{1}\bigcap X_{2}\right)(X_{1}+X_{2})\subseteq X_{2}X_{1} +X_{1}X_{2} = 0$, which implies that $R\cong R/X_{1}\bigoplus R/X_{2}$.
\end{proof}

\begin{rem}\label{1}Suppose that every graded ideal of a graded ring $R$ is graded weakly prime. By Proposition \ref{Proposition 1(1)} and Corollary \ref{Corollary 1(1)}, any nontrivial idempotent graded ideal of $R$ is a graded prime ideal. Recall that the intersection of all graded prime ideals of a graded ring $R$ is called the graded prime radical of $R$. We denote the graded prime radical of $R$ by $GP(R)$, and the sum of all graded ideals whose square is zero by $GN(R)$.
\end{rem}

\begin{thm}\label{Theorem 1(1)}Let $R$ be a graded ring such that every graded ideal of $R$ is graded weakly prime and $R^{2} = R$. Then $GP(R) = GN(R)$ and $(GP(R))^{2} = (GN(R))^{2} = 0$.
\end{thm}

\begin{proof}Let $a, b\in GN(R)$. Then there exist finitely many square zero graded ideals $I_{1}, I_{2}, . . . , I_{k}$ such that $a, b\in I_{1} + I_{2} +...+ I_{k}$. Since $I_{j}^{2}= 0$ for each $j$, $(I_{1} + I_{2} +...+ I_{k})^{m} = 0$ for some $m$, but then $(I_{1} + I_{2} +...+ I_{k})^{2} = 0$ by Corollary \ref{Corollary 1(1)}. Hence, $(GN(R))^{2} = 0$. This implies that if $P$ is any graded prime ideal of $R$, $GN(R)\subseteq P$ and consequently, $GN(R)\subseteq GP(R)$. We should note that $R$ contains at least one graded prime ideal. Indeed, if $R$ contains a nonzero idempotent graded ideal, then by Remark \ref{1}, it should be graded prime. If every graded ideal is nilpotent, then since $R^{2} = R$, $GN(R)\neq R$ is a graded prime ideal. If $GP(R)$ is not a graded prime ideal, then $(GP(R))^{2} = 0$. This implies that $GP(R)\subseteq GN(R)$ by the definition of $GN(R)$, and hence the result follows. Suppose that $GP(R)$ is a graded prime ideal. In this case, we will show that $GN(R)$ should also be graded prime. This implies that $GP(R)\subseteq GN(R)$ by the definition of $GP(R)$, and hence the result follows. Suppose that $IJ\subseteq GN(R)$ for graded ideals $I$ and $J$ of $R$. Since $GN(R)$ is graded weakly prime, we have $J\subseteq GN(R)$ or $I\subseteq GN(R)$ provided that $IJ\neq 0$. Suppose that $IJ = 0$. If $I^{2} = 0$ or $J^{2} = 0$, then $J\subseteq GN(R)$ or $I\subseteq GN(R)$. If both $I$ and $J$ are not square zero, then they are graded prime ideals, but then $either I = I^{2}\subseteq IJ = 0$ or $J = J^{2}\subseteq IJ = 0$, a contradiction. Thus $GN(R)$ is a graded prime ideal, and hence the result follows.
\end{proof}

\begin{cor}\label{Corollary 3(1)}Let $R$ be a graded ring such that every graded ideal of $R$ is graded weakly prime. Then every nonzero graded ideal of $R/GN(R)$ is graded prime.
\end{cor}

\begin{cor}\label{Corollary 4(1)}Let $R$ be a graded ring such that every graded ideal of $R$ is graded weakly prime. Then $(GN(R))^{2} = 0$ and every graded prime ideal contains $GN(R)$. In fact, there are three possibilities:
\begin{enumerate}
\item $GN(R) = R$.

\item $GN(R) = GP(R)$ is the smallest graded prime ideal and all other graded prime ideals are idempotent. If $GN(R)\neq 0$, then it is the only non-idempotent graded prime ideal.

\item $GN(R) = GP(R)$ is not a graded prime ideal.
\end{enumerate}
\end{cor}

\begin{proof}If $R^{2} = 0$, then $GN(R) = R$. So clearly $(GN(R))^{2} = 0$, and there are no graded prime ideals. If $R^{2} = R$, then by Theorem \ref{Theorem 1(1)}, $GP(R) = GN(R)$ and $(GP(R))^{2} = (GN(R))^{2} = 0$. By the definition of $GP(R)$, every graded prime ideal contains $GN(R) = GP(R)$. If $GN(R) = GP(R)$ is graded prime, it is evidently the smallest graded prime ideal and all other graded prime ideals are idempotent.
\end{proof}

\begin{ques}Let $R$ be a graded ring such that every graded ideal of $R$ is graded weakly prime. If $GN(R) = GP(R)$ is not a graded prime ideal, what could we conclude about $GN(R)$?.
\end{ques}

Assume that $M$ is a left $R$-module. Then $M$ is said to be $G$-graded if
$M=\displaystyle\bigoplus_{g\in G}M_{g}$ with $R_{g}M_{h}\subseteq M_{gh}$ for
all $g,h\in G$ where $M_{g}$ is an additive subgroup of $M$ for all $g\in G$.
The elements of $M_{g}$ are called homogeneous of degree $g$. It is clear that
$M_{g}$ is an $R_{e}$-submodule of $M$ for all $g\in G$. We assume that
$h(M)=\displaystyle\bigcup_{g\in G}M_{g}$. Let $N$ be an $R$-submodule of a
graded $R$-module $M$. Then $N$ is said to be graded $R$-submodule if
$N=\displaystyle\bigoplus_{g\in G}(N\cap M_{g})$, i.e., for $x\in N$,
$x=\displaystyle\sum_{g\in G}x_{g}$ where $x_{g}\in N$ for all $g\in G$. It is
known that an $R$-submodule of a graded $R$-module need not be graded.

Let $M$ be a left $R$-module. The idealization $R(+)M=\left\{  (r,m):r\in
R\mbox{ and }m\in M\right\}  $ of $M$ is a ring with componentwise
addition and multiplication; $(x,m_{1})+(y,m_{2})=(x+y,m_{1}+m_{2})$ and
$(x,m_{1})(y,m_{2})=(xy,xm_{2}+ym_{1})$ for each $x,y\in R$ and $m_{1}%
,m_{2}\in M$. Let $G$ be an abelian group and $M$ be a $G$-graded $R$-module.
Then $X=R(+)M$ is $G$-graded by $X_{g}=R_{g}(+)M_{g}$ for all $g\in G$. Note
that, $X_{g}$ is an additive subgroup of $X$ for all $g\in G$. Also, for
$g,h\in G$, $X_{g}X_{h}=(R_{g}(+)M_{g})(R_{h}(+)M_{h})=(R_{g}R_{h},R_{g}%
M_{h}+R_{h}M_{g})\subseteq(R_{gh},M_{gh}+M_{hg})\subseteq(R_{gh}%
,M_{gh})=X_{gh}$ as $G$ is abelian\ \cite{RaTeShKo}.

\begin{lem}\label{5}Let $G$ be an abelian group, $M$ be a $G$-graded $R$-module, $P$ be
an ideal of $R$ and $N$ be an $R$-submodule of $M$ such that $PM\subseteq N$.
Then $P(+)N$ is a graded ideal of $R(+)M$ if and only if $P$ is a graded ideal
of $R$ and $N$ is a graded $R$-submodule of $M$.
\end{lem}

\begin{proof}
Follows from \cite[Proposition 3.3]{RaTeShKo}.
\end{proof}

\begin{exa}\label{Example 5(1)}Let $G$ be an abelian group and $M$ be a $G$-graded $R$-module. Consider the $G$-graded ring $R(+)M$ whose graded ideals are precisely of the form $P(+)N$ where $P$ is a graded ideal of $R$ and $N$ is a graded $R$-submodule of $M$ containing $PM$.
\begin{enumerate}
\item Let $R$ be a graded prime ring that contains exactly one nonzero proper graded ideal $P$. Then every graded ideal of $S_{1} = R(+) P$ is graded weakly prime: the graded maximum ideal $P_{1} = P (+)P$ is idempotent and the nonzero minimal graded ideal $P_{2} = 0 (+)P$ is nilpotent, both of which are graded prime.

\item Every graded ideal of $S_{2} = S_{1}(+)P_{2}$ is graded weakly prime: The graded maximum ideal $Q_{1} =P_{1}(+)P_{2}$ is idempotent and the three nonzero nilpotent graded ideals are $Q_{2} = P_{2}(+)P_{2}$, $Q_{3} = 0(+) P_{2}$, and $Q_{4} = P_{2}(+)0$.
\end{enumerate}
\end{exa}

We introduce an example that gives a nonzero idempotent graded weakly prime right ideal that is not graded prime, and this is unlike the case of graded weakly prime two sided ideals.

\begin{exa}Let $K$ be a field, consider the ring $R=\left(\begin{array}{cc}
                                                            K & K \\
                                                            0 & K
                                                          \end{array}
\right)$ and $G=\mathbb{Z}_{4}$. Then $R$ is $G$-graded by $R_{0}=\left(\begin{array}{cc}
                                                                          K & 0 \\
                                                                          0 & K
                                                                        \end{array}
\right)$, $R_{2}=\left(\begin{array}{cc}
                         0 & K \\
                         0 & 0
                       \end{array}
\right)$ and $R_{1}=R_{3}=0$. Then the graded right ideal $P=\left(\begin{array}{cc}
                                                                         0 & 0 \\
                                                                         0 & K
                                                                       \end{array}
\right)$ of $R$ is graded weakly prime and $P^{2}=P\neq0$. But $P$ is not graded prime right ideal of $R$.
\end{exa}

It should be noted that a proper graded ideal $P$ with property that $P^{2} =\{0\}$ need not be graded weakly prime; see the following example:

\begin{exa}Consider the ring $R=\left(\begin{array}{cc}
                                                            \mathbb{Q} & \mathbb{R} \\
                                                            0 & \mathbb{Q}
                                                          \end{array}
\right)$ and $G=\mathbb{Z}_{4}$. Then $R$ is $G$-graded by $R_{0}=\left(\begin{array}{cc}
                                                                          \mathbb{Q} & 0 \\
                                                                          0 & \mathbb{Q}
                                                                        \end{array}
\right)$, $R_{2}=\left(\begin{array}{cc}
                         0 & \mathbb{R} \\
                         0 & 0
                       \end{array}
\right)$ and $R_{1}=R_{3}=0$. Then the graded ideal $P=\left(\begin{array}{cc}
                                                                         0 & \mathbb{R} \\
                                                                         0 & 0
                                                                       \end{array}
\right)$ of $G(R)$ satisfies $P^{2}=\{0\}$. But $P$ is not graded weakly prime ideal of $G(R)$, since $\left(\begin{array}{cc}
                                                                                                         0 & 0 \\
                                                                                                         0 & 0
                                                                                                       \end{array}
\right)\neq\left(\begin{array}{cc}
                   3 & 0 \\
                   0 & 0
                 \end{array}
\right)\left(\begin{array}{cc}
                                                            \mathbb{Q} & \mathbb{R} \\
                                                            0 & \mathbb{Q}
                                                          \end{array}
\right)\left(\begin{array}{cc}
                                                            0 & 2 \\
                                                            0 & 0
                                                          \end{array}
\right)\subseteq P$.
\end{exa}

\begin{defn}Let $R$ be a graded ring. A non-empty set $S\subseteq h(R)-\{0\}$ is called a graded weakly system if for graded ideals $I$ and $J$ of $R$ with $I\bigcap S\neq\emptyset$, $J\bigcap S\neq\emptyset$ and $IJ\neq0$, we have $IJ\bigcap S\neq\emptyset$.
\end{defn}

\begin{prop}\label{Lemma 3.2}For a proper graded ideal $P$ of $R$, $S = h(R)-P$ is a graded weakly system if and only if $P$ is a graded weakly prime
ideal of $R$.
\end{prop}

\begin{proof}Suppose that $P$ is a graded weakly prime ideal of $R$. Let $I$ and $J$ be graded ideals in $R$ such that $I\bigcap S\neq\emptyset$, $J\bigcap S\neq\emptyset$ and $IJ\neq 0$. If $IJ\bigcap S=\emptyset$, then $IJ\subseteq P$. Since $P$ is graded weakly prime, and $IJ\neq 0$, $I\subseteq P$ or $J\subseteq P$. It follows that $I\bigcap S =\emptyset$ or $J\bigcap S =\emptyset$, which is a contradiction. Therefore, $S$ is a graded weakly system in $R$. Conversely, suppose that $IJ\subseteq P$ and $IJ\neq 0$, where $I$ and $J$ are graded ideals of $R$. If $I\nsubseteq P$ and $J\nsubseteq P$, then $I\bigcap S\neq\emptyset$ and $J\bigcap S\neq\emptyset$. Since $S$ is a graded weakly system, $IJ\bigcap S\neq\emptyset$, which is a contradiction. Therefore, $P$ is a graded weakly prime ideal of $R$.
\end{proof}

\begin{prop}\label{Proposition 3.4}Let $S$ be a graded weakly system, and let $P$ a graded ideal of $R$ maximal with respect to the property that $P$ is disjoint from $S$. Then P is a graded weakly prime ideal.
\end{prop}

\begin{proof}Suppose that $0\neq IJ\subseteq P$, where $I$ and $J$ are graded ideals of $R$. If $I\nsubseteq P$ and $J\nsubseteq P$, then by the maximal property of $P$, we have, $(P +I)\bigcap S\neq\emptyset$ and $(P +J)\bigcap S\neq\emptyset$. Furthermore, $0\neq IJ\subseteq (P +I)(P +J)\subseteq P$. Thus, since $S$ is a graded weakly system, $(P + I)(P + J)\bigcap S\neq\emptyset$ and it follows that $(P + I)(P + J)\nsubseteq P$. For this to happen, we should have $IJ\nsubseteq P$, which is a contradiction. Thus, $P$ should be a graded weakly prime ideal.
\end{proof}

\begin{defn}\label{Definition 3.5}Let $R$ be a graded ring. For a graded ideal $I$ of $R$, if there is a graded weakly prime ideal containing $I$, then we define

\begin{center}$GW(I)= \left\{a\in h(R):\mbox{ every graded weakly system containing }a \mbox{ meets }I\right\}$.\end{center} If there is no graded weakly prime ideal containing $I$, then we put $GW(I) = R$.
\end{defn}

For a graded ideal $I$ of $R$, observe that $I$ and $GW(I)$ are contained in precisely the same graded weakly prime ideals of $R$.

\begin{thm}\label{Theorem 3.6}Let $I$ be a graded ideal of $R$. Then either $GW(I) = R$ or $GW(I)$ equals the intersection of all graded weakly prime ideals of $R$ containing $I$.
\end{thm}

\begin{proof}Suppose that $GW(I)\neq R$. This means that

$\left\{P : P \mbox{ is a graded weakly prime ideal of }R\mbox{ and }I\subseteq P\right\}\neq\emptyset$. We first prove that $GW(I)\subseteq \left\{P : P \mbox{ is a graded weakly prime ideal of }R\mbox{ and }I\subseteq P\right\}$. Let $m\in GW(I)$ and $P$ be any graded weakly prime ideal of $R$ containing $I$. Consider the graded weakly system $h(R)-P$. This graded weakly system cannot contain $m$, for otherwise it meets $I$ and hence also $P$. Therefore, we have $m\in P$. Conversely, assume that $m\notin GW(I)$. Then, by Definition \ref{Definition 3.5}, there exists a graded weakly system $S$ containing $m$ which is disjoint from $I$. By Zorn's Lemma, there exists a graded ideal $P\supseteq I$ which is maximal with respect to being disjoint from $S$. By Proposition \ref{Proposition 3.4}, $P$ is a graded weakly prime ideal of $R$ and we have $m\notin P$, as needed.
\end{proof}

\begin{exa}\label{Example 3.8}Consider the ring $R=\left\{\left(\begin{array}{cc}
                                                                  x & y \\
                                                                  0 & 0
                                                                \end{array}
\right):x, y\in \mathbb{Z}_{4}, b\in \{0, 2\}\right\}$ and $G=\mathbb{Z}_{4}$. Then $R$ is $G$-graded by $R_{0}=\left(\begin{array}{cc}
                                                                                                                        x & 0 \\
                                                                                                                        0 & 0
                                                                                                                      \end{array}
\right)$, $R_{2}=\left(\begin{array}{cc}
                         0 & y \\
                         0 & 0
                       \end{array}
\right)$ and $R_{1}=R_{3}=0$. $R$ has two proper graded ideals $P_{1}=\left\{\left(\begin{array}{cc}
                                                                                              0 & 0 \\
                                                                                              0 & 0
                                                                                            \end{array}
\right), \left(\begin{array}{cc}
                 0 & 2 \\
                 0 & 0
               \end{array}
\right)\right\}$ and $P_{2}=\left\{\left(\begin{array}{cc}
                                                                                              0 & 0 \\
                                                                                              0 & 0
                                                                                            \end{array}
\right), \left(\begin{array}{cc}
                 2 & 0 \\
                 0 & 0
               \end{array}
\right)\right\}$. $P_{1}$ is a graded weakly prime ideal which is not a graded prime ideal since $P_{2}P_{2}=\left(\begin{array}{cc}
                                                                                                                     0 & 0 \\
                                                                                                                     0 & 0
                                                                                                                   \end{array}
\right)\subseteq P_{1}$ but $P_{2}\nsubseteq P_{1}$. $GW(P_{1})=P_{1}$ and $GW(P_{2})=R$.
\end{exa}

\section{Graded Weakly Total Prime Ideals}

In this section, we introduce and study the concept of graded weakly total prime ideals. Recall that in \cite{Dawwas Bataineh Muanger}, a proper graded ideal $P$ of $R$ is a graded total prime ideal if $xy\in P$ implies $x\in P$ or $y\in P$ for $x, y\in h(R)$.

\begin{defn}\begin{enumerate}
\item A proper graded ideal $P$ of $R$ is graded weakly total prime if $0\neq xy\in P$ implies $x\in P$ or $y\in P$ for $x, y\in h(R)$.

\item Let $R$ be a $G$-graded ring, $P$ be a graded ideal of $R$ and $g\in G$ such that $P_{g}\neq R_{g}$. Then $P$ is said to be a $g$-weakly total prime ideal of $R$ if whenever $x, y\in R_{g}$ such that $0\neq xy\in P$, then either $x\in P$ or $y\in P$.
\end{enumerate}
\end{defn}

The following example shows that not every graded weakly prime ideal is a graded weakly total prime ideal.

\begin{exa}\label{Example 4.11}Consider the ring of $2\times2$ matrices with integer entries $R=M_{2}(\mathbb{Z})$ and $G=\mathbb{Z}_{4}$. Then $R$ is $G$-graded by $R_{0}=\left(\begin{array}{cc}
                  \mathbb{Z} & 0 \\
                  0 & \mathbb{Z}
                \end{array}
\right)$, $R_{2}=\left(\begin{array}{cc}
                  0 & \mathbb{Z} \\
                  \mathbb{Z} & 0
                \end{array}
\right)$ and $R_{1}=R_{3}=0$. Consider the graded ideal $P=M_{2}(2\mathbb{Z})$ of $R$. Clearly, $P$ is a graded prime ideal and hence also graded weakly prime ideal of $R$. On the hand, $P$ is not graded weakly total prime since $\left(\begin{array}{cc}
                  6 & 0 \\
                  0 & 3
                \end{array}
\right)\left(\begin{array}{cc}
                  0 & 1 \\
                  2 & 0
                \end{array}
\right)=\left(\begin{array}{cc}
                  0 & 6 \\
                  6 & 0
                \end{array}
\right)\in P$.
\end{exa}

The following three examples shows that not every graded weakly total prime ideal is a graded total prime ideal.

\begin{exa}\label{Example 4.12}Consider the ring $R=\left(\begin{array}{cc}
                  \mathbb{Z} & 2\mathbb{Z} \\
                  0 & \mathbb{Z}
                \end{array}
\right)$ and $G=\mathbb{Z}_{4}$. Then $R$ is $G$-graded by $R_{0}=\left(\begin{array}{cc}
                  \mathbb{Z} & 0 \\
                  0 & \mathbb{Z}
                \end{array}
\right)$, $R_{2}=\left(\begin{array}{cc}
                  0 & 2\mathbb{Z} \\
                  0 & 0
                \end{array}
\right)$ and $R_{1}=R_{3}=0$. Clearly, $\left\{\left(\begin{array}{cc}
                  0 & 0 \\
                  0 & 0
                \end{array}
\right)\right\}$ is a graded weakly total prime ideal of $R$ which is not graded total prime since $\left(\begin{array}{cc}
                  0 & 2 \\
                  0 & 0
                \end{array}
\right)\left(\begin{array}{cc}
                  0 & 2 \\
                  0 & 0
                \end{array}
\right)=\left(\begin{array}{cc}
                  0 & 0 \\
                  0 & 0
                \end{array}
\right)$.
\end{exa}

\begin{exa}\label{Example 4.13}the graded ideal $P_{1}$ in Example \ref{Example 3.8} is a graded weakly total prime ideal which is not graded total prime since $\left(\begin{array}{cc}
                  2 & 0 \\
                  0 & 0
                \end{array}
\right)\left(\begin{array}{cc}
                  2 & 0 \\
                  0 & 0
                \end{array}
\right)=\left(\begin{array}{cc}
                  0 & 0 \\
                  0 & 0
                \end{array}
\right)\in P_{1}$.
\end{exa}

\begin{exa}\label{Example 4.14}Let $M$ be a $G$-graded left $R$-module. Let $Z(R)$ be the set of all zero-divisors of $R$ and $(0 :_{R} M) =\left\{a\in R : aM = 0\right\}$ the annihilator of $M$ in $R$. Suppose that $0\neq Z(R)\subseteq (0 :_{R} M)$. Let $[R, M] =\left\{(a, m) : a\in R \mbox{ and }m\in M\right\}$ be the ring with componentwise addition and multiplication $(a, m)(b, n) = (ab, an)$. In fact, $[R, M]$ is $G$-graded ring by $[R, M]_{g}=[R_{g}, M_{g}]$ for all $g\in G$, one can observe that $[R, M]_{g}[R, M]_{h}=[R_{g}, M_{g}][R_{h}, M_{h}]=[R_{g}R_{h}, R_{g}M_{h}]\subseteq[R_{gh}, M_{gh}]=[R, M]_{gh}$ for all $g, h\in G$. Now $[0, M]$ is a graded ideal of $[R, M]$. In fact, it is a graded weakly total prime ideal, but not a graded total prime ideal.
\end{exa}

\begin{defn}\begin{enumerate}
\item Let $R$ be a graded ring, $x, y\in h(R)$ and $P$ be a graded weakly total prime ideal of $R$. We say that $(x, y)$ is a total homogeneous twin-zero of $P$ if $xy = 0$, $x\notin P$ and $y\notin P$.
    
\item Let $R$ be a graded ring, $P$ be a $g$-weakly total prime ideal of $R$ and $x, y\in R_{g}$. We say that $(x, y)$ is a $g$-total twin-zero of $P$ if $xy = 0$, $x\notin P$ and $y\notin P$.
\end{enumerate}
\end{defn}

Note that if $P$ is a graded weakly total prime (a $g$-weakly total prime) ideal of $R$ that is not a graded total prime (not a $g$-total prime) ideal, then $P$ has a total homogeneous twin-zero $(x, y)$ for some $x, y\in h(R)$ (a $g$-total twin-zero $(x, y)$ for some $x, y\in R_{g}$).

\begin{lem}\label{Lemma 4.3} Let $P$ be a $g$-weakly total prime ideal of $R$ and suppose that $(x, y)$ is a $g$-total twin-zero of $P$. Then $xP_{g} = P_{g}y = 0$.
\end{lem}

\begin{proof}Suppose that $xP_{g}\neq 0$. Then there exists $p\in P_{g}$ such that $xp\neq0$. Hence $x(y + p)\neq 0$. Since $x\notin P$ and $P$ is $g$-weakly total prime, we have $x + p\in P$, and hence $x\in P$, which is a contradiction. Thus $xP_{g} = 0$. Similarly, it can be easily verified that $P_{g}y = 0$.
\end{proof}

\begin{lem}\label{Lemma 4.4} Let $P$ be a $g$-weakly total prime ideal of $R$ and suppose that $(x, y)$ is a $g$-total twin-zero of $P$. If $xr\in P$ for some $r\in R_{g}$, then $xr = 0$.
\end{lem}

\begin{proof}Suppose that $0\neq xr$. Then $0\neq xr\in P$, and then $r\in P$ since $P$ is $g$-weakly total prime and $(x, y)$ is a $g$-total twin-zero of $P$. Now, since $xr\in xP$, we have that $xr = 0$ by Lemma \ref{Lemma 4.3}, which is a contradiction.
\end{proof}

\begin{thm}\label{Theorem 4.5} Let $R$ be a $G$-graded ring and $g\in G$. If $P$ is $g$-weakly total prime ideal of $R$ but not $g$-total prime, then $P_{g}^{2} = 0$.
\end{thm}

\begin{proof}Let $(x, y)$ be a $g$-total twin-zero of $P$. Suppose that $pq\neq 0$ for some $p, q\in P_{g}$. Then by Lemma \ref{Lemma 4.3}, we have $0\neq (x + p)(y + q)\in P$. Thus $(x + p)\in P$ or $(y + q)\in P$ and hence $x\in P$ or $y\in P$, which is a contradiction. Therefore $P_{g}^{2} = 0$.
\end{proof}

\begin{cor}\label{Corollary 4.6} Let $R$ be a $G$-graded ring and let $P$ a graded ideal of $R$. If $P_{g}^{2}\neq 0$ for some $g\in G$, then $P$ is a $g$-total prime ideal of $R$ if and only if $P$ is a $g$-weakly total prime ideal of $R$.
\end{cor}

It should be noted that a proper graded ideal $P$ with property that $P^{2} =\{0\}$ need not be graded weakly total prime; see the following example:

\begin{exa}Consider the ring $R=\left(\begin{array}{cc}
                                                            \mathbb{Q} & \mathbb{R} \\
                                                            0 & \mathbb{Q}
                                                          \end{array}
\right)$ and $G=\mathbb{Z}_{4}$. Then $R$ is $G$-graded by $R_{0}=\left(\begin{array}{cc}
                                                                          \mathbb{Q} & 0 \\
                                                                          0 & \mathbb{Q}
                                                                        \end{array}
\right)$, $R_{2}=\left(\begin{array}{cc}
                         0 & \mathbb{R} \\
                         0 & 0
                       \end{array}
\right)$ and $R_{1}=R_{3}=0$. Then the graded ideal $P=\left(\begin{array}{cc}
                                                                         0 & \mathbb{R} \\
                                                                         0 & 0
                                                                       \end{array}
\right)$ of $R$ satisfies $P^{2}=\{0\}$. But $P$ is not graded weakly total prime ideal of $R$, since $\left(\begin{array}{cc}
                                                                                                         0 & 0 \\
                                                                                                         0 & 0
                                                                                                       \end{array}
\right)\neq\left(\begin{array}{cc}
                   3 & 0 \\
                   0 & 0
                 \end{array}
\right)\left(\begin{array}{cc}
                                                            0 & 2 \\
                                                            0 & 0
                                                          \end{array}
\right)=\left(\begin{array}{cc}
                       0 & 6 \\
                       0 & 0
                     \end{array}\right)
\in P$.
\end{exa}

\begin{prop}\label{Proposition 4.9}Let $P$ be a $g$-weakly total prime ideal of $R$. If $x\in R_{g}$ and $Y\subseteq R_{g}$ such that $0\neq xY\subseteq P$, then either $x\in P$ or $Y\subseteq P$.
\end{prop}

\begin{proof}Suppose that $x\notin P$. For every $b\in Y$ such that $0\neq xb\in P$, we have $b\in P$ since $P$ is $g$-weakly total prime. If $y\in Y$ such that $0 = xy\in P$ and $y\notin P$, then $(x, y)$ is a $g$-total twin-zero of $P$. Because $xY\subseteq P$, it follows by Lemma \ref{Lemma 4.4} that $xY = 0$, which is a contradiction and therefore $y\in P$ and we have $Y\subseteq P$.
\end{proof}

\begin{thm}\label{Theorem 4.21}For a graded ideal $P$ of $R$ and $g\in G$ such that $P_{g}\neq R_{g}$, the following statements are equivalent:
\begin{enumerate}
\item $P$ is a $g$-weakly total prime ideal of $R$.

\item For any subset $Y$ of $R_{g}$ such that $Y\nsubseteq P$, $(P :_{R_{g}} Y) = \left\{x\in R_{g} : xY\subseteq P\right\} = P\bigcup (0 :_{R_{g}} Y)$.

\item For any subset $Y$ of $R_{g}$ such that $Y\nsubseteq P$, $(P :_{R_{g}} Y) = P$ or $(P :_{R_{g}} Y) = (0 :_{R_{g}} Y)$.
\end{enumerate}
\end{thm}

\begin{proof}$(1)\Rightarrow(2)$: Let $x\in (P :_{R_{g}} Y)$. Now $xY\subseteq P$. If $xY\neq 0$, then since $P$ is $g$-weakly total prime it follows by Proposition \ref{Proposition 4.9} that $x\in P$. If $xY = 0$, then $x\in (0 :_{R_{g}} Y)$. So, $(P :_{R_{g}} Y)\subseteq P\bigcup (0 :_{R_{g}} Y)$. As the reverse containment holds for any graded ideal $P$, we have equality.

$(2)\Rightarrow(3)$: Since $P$ and $(0 :_{R_{g}} Y)$ are both subgroups of $R$, it follows that either $(P :_{R_{g}} Y) = P$ or $(P :_{R_{g}} Y) = (0 :_{R_{g}} Y)$.

$(3)\Rightarrow(1)$: Let $x, y\in R_{g}$ such that $0\neq xy\in P$. If $y\in P$, then we are done. So, suppose that $y\in R_{g}-P$. Then $(P :_{R_{g}} y)\neq (0 :_{R_{g}} y)$ and from (3), we have $(P :_{R_{g}} y) = P$. Hence, $x\in P$ and we are done.
\end{proof}

\begin{cor}\label{Corollary 4.22}For a graded ideal $P$ of $R$ and $g\in G$ such that $P_{g}\neq R_{g}$, the following statements are equivalent:
\begin{enumerate}
\item $P$ is a $g$-weakly total prime ideal of $R$.

\item For $x\in R_{g}-P$, $(P :_{R_{g}} x) = P\bigcup (0 :_{R_{g}} x)$.

\item For $x\in R_{g}-P$, $(P :_{R_{g}} x) = P$ or $(P :_{R_{g}} x) = (0 :_{R_{g}} x)$.
\end{enumerate}
\end{cor}

Recall that in \cite{Nastasescue}, if $T$ and $L$ are two $G$-graded rings, then $R=T\times L$ is a $G$-graded ring by $R_{g}=T_{g}\times L_{g}$ for all $g\in G$.

\begin{thm}\label{Theorem 4.24}Let $T$ and $L$ be two $G$-graded rings with unities and $R=T\times L$. If $P$ is a graded weakly total prime ideal of $R$, then either $P = 0$ or $P$ is a graded total prime ideal of $R$.
\end{thm}

\begin{proof}Assume that $P=P_{1}\times P_{2}$ where $P_{1}$ is a graded ideal of $T$ and $P_{2}$ is a graded ideal of $L$. Suppose that $P\neq0$. Then there is an element $(x, y)$ of $P$ such that $(x, y)\neq (0, 0)$, and then there exists $g\in G$ such that $(x_{g}, y_{g})=(x, y)_{g}\neq(0, 0)$. As $P_{1}$ and $P_{2}$ are graded, $x_{g}\in P_{1}$ and $y_{g}\in P_{2}$. Now, $(0, 0)\neq (x_{g}, y_{g}) = (x_{g}, 1)(1, y_{g})\in P$ and $P$ is graded weakly total prime gives that $(x_{g}, 1)\in P$ or $(1, y_{g})\in P$. Suppose that $(x_{g}, 1)\in P$. Then $0\times L\subseteq P$, so $P = P_{1}\times L$. We show that $P_{1}$ is a graded total prime ideal of $T$. Let $pq\in P_{1}$, where $p, q\in h(T)$. Then $(0, 0)\neq (pq, 1) = (p, 1)(q, 1)\in P$. Now $P$ is graded weakly total prime gives $(p, 1)\in P$ or $(q, 1)\in P$. Hence, $p\in P_{1}$ or $q\in P_{1}$. So, $P_{1}$ is a graded total prime ideal of $T$. The case $(1, y)\in P$ is similar.
\end{proof}

Recall that in \cite{Nastasescue}, if $R$ is a $G$-graded ring and $I$ is a graded ideal of $R$, then $R/I$ is a $G$-graded ring by $(R/I)_{g}=(R_{g}+I)/I$ for all $g\in G$.

\begin{prop}\label{Proposition 4.25}Let $R$ be a graded ring and $I$, $P$ be proper graded ideals of $R$ such that $I\subseteq P$. Then the following holds:
\begin{enumerate}
\item If $P$ is a graded weakly total prime ideal of $R$, then $P/I$ is a graded weakly total prime ideal of $R/I$.

\item If $I$ is a graded weakly total prime ideal of $R$ and $P/I$ is a graded weakly total prime ideal of $R/I$, then $P$ is a graded weakly total prime ideal of $R$.
\end{enumerate}
\end{prop}

\begin{proof}
\begin{enumerate}
\item Let $0\neq (x + I)(y + I) = (xy + I)\in P/I$ where $x, y\in h(R)$, so $xy\in P$. If $xy = 0\in I$, then $(x+I)(y+I) = 0$, which is a contradiction. Hence, $xy\neq 0$ and since $xy\in P$ and P is graded weakly total prime, we get $x\in P$ or $y\in P$. Hence, $(x + I)\in P/I$ or $(y + I)\in P/I$ as required.
    
\item Let $0\neq xy\in P$ where $x, y\in h(R)$, so that $(x+I)(y+I)\in P/I$. If $xy\in P$, then since $I$ is graded weakly total prime, we get $x\in I\subseteq P$ or $y\in I\subseteq P$. If $xy\notin I$, then $0\neq (x + I)(y + I)\in P/I$. Now, since $P/I$ is graded weakly total prime, we get $(x + I)\in P/I$ or $(y + I)\in P/I$. Hence, $x\in P$ or $y\in P$ as needed.
\end{enumerate}
\end{proof}

\begin{cor}\label{Theorem 4.26}Let $P_{1}$ and $P_{2}$ be graded weakly total prime ideals of $R$ that are not graded total prime. Then $P_{1} + P_{2}$ is a graded weakly total prime ideal of $R$.
\end{cor}

\begin{proof}Since $(P_{1} +P_{2})/P_{2}\cong P_{2}/\left(P_{1}\bigcap P_{2}\right)$we get that $(P_{1} +P_{2})/P_{2}$ is graded weakly total prime by Proposition \ref{Proposition 4.25} (1). Now the assertion follows from Proposition \ref{Proposition 4.25} (2).
\end{proof}

We are interested in the structure of graded rings in which every graded ideal is graded weakly total prime. Note that by definition, a graded weakly total prime ideal is a proper graded ideal of a graded ring. It is therefore not possible that every graded ideal of a graded ring is a graded weakly total prime ideal. However, a graded ring whose zero ideal is graded total prime is called a graded total prime ring. In this sense, every graded ring is a graded weakly total prime ring since the zero ideal is always graded weakly total prime. We may therefore say that every graded ideal of a graded ring is graded weakly total prime when every proper graded ideal of the graded ring is a graded weakly total prime ideal. If $R^{2} = 0$, then it is evident that every graded ideal of $R$ is graded weakly total prime.

Recall that, for a ring $R$ and $x\in R$, 

\begin{center}$\langle x\rangle=\left\{\displaystyle\sum_{i=1}^{n}r_{i}xs_{i}+rx+xs+mx:n\in \mathbb{N}, m\in \mathbb{Z}, r_{i}, s_{i}, r, s\in R\right\}$.\end{center}
If $R$ has a unity, then $\langle x\rangle=\left\{\displaystyle\sum_{i=1}^{n}r_{i}xs_{i}:n\in \mathbb{N}, r_{i}, s_{i}\in R\right\}$.

\begin{prop}\label{Proposition 5.1}Every graded ideal of $R$ is graded weakly total prime if and only if for every $a, b\in h(R)$ we have $\langle ab\rangle = \langle a\rangle$, $\langle ab\rangle = \langle b\rangle$ or $\langle ab\rangle = 0$.
\end{prop}

\begin{proof}Suppose that every graded ideal of $R$ is a graded weakly total prime ideal and let $a, b\in h(R)$. If $\langle ab\rangle = R$, then $\langle a\rangle = \langle b\rangle = R$. Suppose that $\langle ab\rangle\neq R$ and $\langle ab\rangle\neq 0$, then $\langle ab\rangle$ is a graded weakly total prime ideal. Now, since $0\neq ab\in \langle ab\rangle$, we have $a\in \langle ab\rangle$ or $b\in \langle ab\rangle$. Hence, $\langle ab\rangle = \langle a\rangle$ or $\langle ab\rangle = \langle b\rangle$. Conversely, let $I$ be any proper graded ideal of $R$ and suppose that $0\neq ab\in I$ for $a, b\in h(R)$. Now we have $\langle a\rangle = \langle ab\rangle\subseteq I$ or $\langle b\rangle = \langle ab\rangle\subseteq I$. Hence, $a\in I$ or $b\in I$ and we are done.
\end{proof}

\begin{cor}\label{Corollary 5.2}Let $R$ be a graded ring in which every graded ideal is graded weakly total prime. Then for every $x\in h(R)$, $\langle x^{2}\rangle=\langle x\rangle$ or $\langle x^{2}\rangle=0$.
\end{cor}

The next example shows that the converse of Corollary \ref{Corollary 5.2} is not true in general.

\begin{exa}Let $K$ be a field and $R = K\bigoplus K\bigoplus K$. Then for every element $x\in R$, we have $\langle x^{2}\rangle =\langle x\rangle$. On the other hand, if we consider the trivial graduation for $R$ by a group $G$, then $P=K\bigoplus0\bigoplus0$ will be a graded ideal of $R$ which is not graded weakly total prime.
\end{exa}

\section*{\textbf{Acknowledgement}}

This research was funded by the Deanship of Scientific Research at Princess Nourah bint Abdulrahman University through the Fast-track Research Funding Program.

\end{document}